\documentclass[12pt]{amsart}
\usepackage[utf8]{inputenc}

\usepackage{microtype,amsmath,amsthm,amssymb}
\usepackage{parskip}
\usepackage{geometry}
 \usepackage[hidelinks]{hyperref}
\author{Paul Pollack} 
\address{Department of Mathematics \\ University of Georgia \\ Athens, GA 30602}
\email{pollack@uga.edu}
\thanks{P.P. is supported by NSF award DMS-2001581.}
 

\author{Akash Singha Roy}
\address{ESIC Staff Quarters No.: D2\\ 143 Sterling Road, Nungambakkam\\Chennai 600034\\ Tamil Nadu, India.}
\email{akash01s.roy@gmail.com}

\renewcommand\phi\varphi

\renewcommand{\pod}[1]{\allowbreak\mathchoice
  {\if@display \mkern 18mu\else \mkern 8mu\fi (#1)}
  {\if@display \mkern 18mu\else \mkern 8mu\fi (#1)}
  {\mkern4mu(#1)}
  {\mkern4mu(#1)}
}
\usepackage{graphicx}
\DeclareMathAlphabet{\curly}{U}{rsfs}{m}{n}
\newcommand{\ff}{\mathbf{f}}
\newcommand\Z{\mathbb{Z}}
\newtheorem{thm}{Theorem}[section]

\newtheorem{prop}[thm]{Proposition}
\newtheorem{lem}[thm]{Lemma}

\theoremstyle{remark}
\newtheorem*{rmk}{Remark}

\newcommand\ord{\mathrm{ord}}
\begin{document}
\title[Polynomial-like multiplicative functions]{Joint distribution in residue classes of polynomial-like multiplicative functions}
\begin{abstract} Under fairly general conditions, we show that families of integer-valued polynomial-like multiplicative functions are uniformly distributed in coprime residue classes mod $p$, where $p$ is a growing prime (or nearly prime) modulus. This can be seen as complementary to work of Narkiewicz, who obtained comprehensive results for fixed moduli. \end{abstract}
\maketitle

\section{Introduction}
%\cite{LPR21}
For any integer-valued arithmetic function, it is reasonable to ask how the values of $f$ are distributed in arithmetic progressions. As stated, this problem is far too general; to get any traction, it is necessary to restrict $f$. Let us suppose that $f$ is multiplicative and that $f$ is \textsf{polynomial-like}, in the sense that there is a polynomial $F(T) \in \Z[T]$ such that $f(p) = F(p)$ for every prime number $p$. In this case, Narkiewicz (beginning in \cite{narkiewicz66}) has made a comprehensive study of the distribution of $f$ in coprime residue classes. For a thorough survey of this work, see Chapter V in \cite{narkiewicz84}. See also \cite{narkiewicz12} for a more recent contribution to this subject by the same author.

In 1982, Narkiewicz \cite{narkiewicz82} observed that his methods could be applied to study the joint distribution of several functions. We state a special case of the main theorem of \cite{narkiewicz12}. Let $f_1, \dots, f_K$ be a finite sequence of multiplicative, integer-valued arithmetic functions. Say that $f_1,\dots,f_K$ is \textsf{nice} if the following conditions hold:
\begin{enumerate}
    \item[(i)] Each $f_k$ is polynomial-like for a nonconstant polynomial: There is a nonconstant polynomial $F_k(T) \in \Z[T]$ such that $f_k(p)=F_k(p)$ for all primes $p$,
\end{enumerate}
and
\begin{enumerate}
\item[(ii)] $F_1(T) \cdots F_{K}(T)$ has no multiple roots.
\end{enumerate}

If $f_1,\dots,f_K$ is a nice family, a prime $p$ is called \textsf{good} for $f_1,\dots,f_K$ if  (a) $p>5$, (b) $p > (1+\sum_{k}\deg{F_k}(T))^2$, (c) $p$ does not divide the leading coefficient of any $F_k(T)$, and (d) $p$ does not divide the discriminant of $F_1(T)\cdots F_K(T)$. For any fixed nice family $f_1,\dots, f_K$, all but finitely many primes are good.
Narkiewicz proves that if every prime divisor of $q$ is good, and one restricts attention to $n$ for which the values $f_1(n),\dots,f_K(n)$
are coprime to $q$, then those values are asymptotically jointly uniformly distributed among the coprime residue classes modulo $q$. More precisely: For every choice of integers $a_1,\dots,a_K$ coprime to $q$, we have \begin{equation}\label{eq:joint}\sum_{\substack{n \le x \\ (\forall k)~f_k(n)\equiv a_k\pmod{q}}} 1 \sim \frac{1}{\phi(q)^{K}} \sum_{\substack{n \le x \\ \gcd(\prod_{k=1}^{K} f_k(n),q)=1}} 1,\end{equation}
as $x\to\infty$. (It is proved along the way that the right-hand side of \eqref{eq:joint} tends to infinity under the same hypotheses.) In particular, we get joint uniform distribution in coprime residue classes mod $p$ for all good primes $p$.

So far everything that has been said concerns the distribution to a fixed modulus $q$. It is natural to also consider the distribution when $q$ grows with $x$. We prove a joint uniform distribution result of this kind for nice families valid when the modulus $q$ is prime or ``nearly prime''. Here ``nearly prime'' means that $\delta(q)$ is small where
\[ \delta(q) := \sum_{p \mid q}\ \frac{1}{p}. \]

Our main theorem is as follows.

\begin{thm}\label{thm:joint} Fix a nice sequence $f_1, \dots, f_K$ of multiplicative functions and fix $\epsilon > 0$. Suppose that $q, x\to\infty$ with $\delta(q)=o(1)$ and $q\le (\log{x})^{\frac{1}{K}-\epsilon}$. Then \eqref{eq:joint} holds, uniformly in the choice of coprime residue classes $a_1,\dots,a_{K}$ mod $q$.
\end{thm}

For example, let $f_1(n)=n$, $f_2(n)=\phi(n)$, and $f_3(n)=\sigma(n)$.  These form a nice family. By the result of Narkiewicz quoted above, the values of $n$, $\phi(n)$, $\sigma(n)$ coprime to $p$ are uniformly distributed in coprime residue classes mod $p$ for each fixed $p \ge 17$. It then follows from Theorem \ref{thm:joint} that this equidistribution holds uniformly for $17 \le p \le (\log{x})^{\frac13-\epsilon}$.

There are two  directions in which one might hope to strengthen Theorem \ref{thm:joint}. First, it would be desirable to weaken the condition $\delta(q)=o(1)$, e.g., by replacing it with Narkiewicz's condition that $q$ is divisible only by good primes. Such an improvement would seem to require a substantial new idea. Second, one might hope to enlarge the range of allowable $q$ past $(\log{x})^{\frac1K-\epsilon}$. It was proved in \cite{LPR21} that when $K=1$ and $f_1(n)=\phi(n)$, one can replace $(\log{x})^{1-\epsilon}$ with $(\log{x})^{A}$, for an arbitrary $A$, provided $q$ is restricted to primes. This might seem to suggest that  $(\log{x})^{\frac{1}{K}-\epsilon}$ in Theorem \ref{thm:joint} can  always be replaced with $(\log{x})^{A}$, with $A$ arbitrary. As we now explain, this is  too optimistic.

Suppose that $f_1,\dots,f_K$ is a fixed nice family with $K\ge 2$. Fix a prime $p_0$ with $f_1(p_0),\dots,f_K(p_0)$ all nonzero. Let $X:= 2(\log{x})^{\frac{1}{K-1}}$, and choose $p$ to be a prime in $(2X/3, X]$. As $x\to\infty$, there are at ``obviously'' at least $(1+o(1)) x/p\log{x} \ge (\frac{4}{3}+o(1)) x/p^K$ values of $n\le x$ having $f_k(n) \equiv f_k(p_0)\pmod{p}$ for all $k=1,\dots,K$, since $n$ can be taken as any prime congruent to $p_0\pmod{p}$. This shows that equidistribution in coprime residue cannot hold up to $X$. It is conceivable that in Theorem \ref{thm:joint} uniformity holds up to $(\log{x})^{\frac{1}{K-1}-\epsilon}$ (interpreted as $(\log{x})^{A}$, $A$ arbitrary, when $K=1$). Again, it would seem to require a new idea to decide this.

% Consider the nice family $n, \sigma(n)$. Suppose that $(\log{x})^{\frac12-\epsilon}$ in the conclusion of Theorem \ref{thm:joint} can be replaced by $X:=3\log{x}$. Let $p$ be a prime with $p \in (X/2,X]$. By uniform distribution, there are  $(1+o(1))x/p^2 \le (2/3+o(1))x/p\log{x}$ values of $n\le x$ with $n\equiv 1\pmod{p}$ and $\sigma(n)\equiv 2\pmod{p}$ (as $x\to\infty$). On the other hand, there are ``obviously'' at least $(1+o(1))x/p\log{x}$ such $n\le x$, since we can take for $n$ any prime congruent to $1\pmod{p}$.  So $(\log{x})^{\frac12-\epsilon}$ cannot be replaced with $3\log{x}$, even restricting $q$ to primes. By a more complicated argument (borrowing ideas from the proof of Theorem 1.3 of \cite{LPR21}), one can show uniformity fails already at $\eta\log{x}$, where $\eta>0$ is arbitrary. It is conceivable that uniform distribution holds when $p= o(\log{x})$; again, it would seem to require a new idea to decide this.

We conclude this introduction with a brief summary of the proof of Theorem \ref{thm:joint}: Split off the first several largest prime factors of $n$, say $n = m P_J \cdots P_1$, where $P^{+}(m) \le P_J \le \dots \le P_1$. (Here $J$ must be chosen judiciously; we also ignore $n$ with fewer than $J$ prime factors.) Most of the time, $P_J, \dots, P_1$ will appear to the first power only in $n$, so that $f_k(n) = f_k(m) f_k(P_J) \cdots f_k(P_1)$. Then given $m$, we use the prime number theorem for progressions (Siegel--Walfisz) and character sum estimates to understand the number of choices for $P_1, \dots, P_J$ compatible with the congruence conditions on $f_k(n)$. 

\subsection*{Notation and conventions} Throughout, the letters $p,P,r$, with or without subscripts, always denote primes whether or not this is explicitly mentioned. We use $P^{+}(n)$ for the largest prime factor of $n$, with the convention that $P^{+}(1)=1$. We write  $\mathfrak{f}(\chi)$ for the conductor of the Dirichlet character $\chi$.

\section{Preparation}
\subsection{Sieve lemmas} We will make frequent use of the following special case of the fundamental lemma of sieve theory, as formulated  in \cite[Theorem 7.2, p.\ 209]{HR74}.

\begin{lem}\label{lem:sievelem} Let $X \ge Z \ge 3$. Suppose that the interval $\mathcal{I} = (u,v]$ has length $v-u=X$. Let $\mathcal{P}$ be a set of primes not exceeding $Z$. For each $p \in \mathcal{P}$, choose a residue class $a_p \bmod{p}$.  The number of integers $n\in \mathcal{I}$ not congruent to $a_p\bmod{p}$ for any $p \in \mathcal{P}$ is 
\[ X \left(\prod_{p\in \mathcal{P}}\left(1-\frac{1}{p}\right)\right)\left(1 + O\left(\exp\left(-\frac{1}{2}\frac{\log{X}}{\log{Z}}\right)\right)\right). \]
\end{lem}

The following application of Lemma \ref{lem:sievelem} yields
a lower bound for the ``numerator'' on the right-hand side of \eqref{eq:joint}. See Scourfield's Theorem 4 in \cite{scourfield84} for a closely related result (and compare with \cite{scourfield85}).

% \begin{lem}\label{lem:coprimeasymp} Fix a nice arithmetic function $f$. Suppose that $q, x\to\infty$ with $q= x^{o(1)}$ and $\delta(q)=o(1)$. The number of $n\le x$ for which $\gcd(f(n),q)=1$ is asymptotically equal to
% \begin{equation}\label{eq:coprimeasymp} x\prod_{\substack{p \le x \\ \gcd(f(p),q)>1}} \left(1-\frac{1}{p}\right).
% \end{equation}
% \end{lem}

\begin{lem}\label{lem:coprimeasymp} Fix a nice arithmetic function $f$ {\rm (}meaning that $f$ is nice when viewed as a singleton sequence{\rm )}. Suppose that $q, x\to\infty$ with $q= x^{o(1)}$ and $\delta(q)=o(1)$. The number of $n\le x$ for which $\gcd(f(n),q)=1$ eventually\footnote{meaning whenever $q, x$ are sufficiently large and $\frac{\log{q}}{\log{x}}, \delta(q)$ are sufficiently small}  exceeds
\begin{equation}\label{eq:coprimeasymp} \frac{1}{20}x\prod_{\substack{p \le x \\ \gcd(f(p),q)>1}} \left(1-\frac{1}{p}\right).
\end{equation}
\end{lem}

\begin{rmk}\mbox{ }
\begin{enumerate}
    \item[(a)] With a small amount of additional effort, one could show that \eqref{eq:coprimeasymp} is the correct order of magnitude for this count of $n$. But we will not need this.
    \item[(b)] It will be useful momentarily to know that the product on $p$ in \eqref{eq:coprimeasymp} has size at least $(\log{x})^{o(1)}$. To see this, choose $F(T) \in \Z[T]$ with $f(p)=F(p)$ for all $p$. It suffices to show that 
    \[ \sum_{\substack{p \le x \\ \gcd(f(p),q) > 1}} 1/p = o(\log\log{x}). \]
    Let $\mathcal{S}$ be the set of primes $p\le x$ with $\gcd(f(p),q)>1$. For each prime $r$ dividing $q$, let $\mathcal{S}_r = \{p \in (r,x]: F(p)\equiv 0\pmod{r}\}$. Since $F$ has $O_f(1)$ roots modulo every prime $r$,
    \[ \sum_{r\mid q} \sum_{p \in \mathcal{S}_r} \frac{1}{p} \ll_{f} \log\log{x}\sum_{r\mid q} \frac{1}{r} = \delta(q) \log\log x = o(\log\log{x}). \]
Here the sum on $p\in \mathcal{S}_r$ has been estimated by partial summation and the Brun--Titchmarsh inequality. For each $r$ dividing $q$, there are $O_f(1)$ primes $p \le r$ with $F(p) \equiv 0 \pmod{r}$. So if we put $\mathcal{S}':=\mathcal{S}\setminus\cup_{r\mid q}\mathcal{S}_r$, then $\#\mathcal{S}' \ll_{f} \omega(q)$, and, writing $p_k$ for the $k$th prime in the usual increasing order,
\[ \sum_{p \in \mathcal{S}'} \frac{1}{p} \le \sum_{k=1}^{\#\mathcal{S}'}\frac{1}{p_k} \ll_{f}  \log\log(3\omega(q)) = o(\log\log{x}), \]
using the simple bound $\omega(q) = O(\log{x})$ in the last step.
\end{enumerate}
\end{rmk}

\begin{proof}[Proof of Lemma \ref{lem:coprimeasymp}] Fix a real number $U\ge 2$. We start by considering all $n\le x$ not divisible by any $p\le x^{1/U}$ with $\gcd(f(p),q)>1$. For large $q, x$ and small $\frac{\log{q}}{\log{x}}, \delta(q)$, where here and below ``large'' and ``small'' may depend on $U$, the sieve shows that the count of such $n$ is
\[ x \bigg(\prod_{\substack{p \le x^{1/U} \\ \gcd(f(p),q)>1}}\left(1-\frac{1}{p}\right)\bigg)(1+ O(\exp(-U/2))). \]
We proceed to bound from above the number of these  $n$ for which $\gcd(f(n),q)>1$. 

For each $n$ surviving our initial sieve but having $\gcd(f(n),q)>1$, we factor $n=A_1 A_2 B$, where
\[ A_1 = \prod_{\substack{p\parallel n \\ \gcd(f(p),q)> 1}} p, \quad A_2 = \prod_{\substack{p^e\parallel n,~e>1 \\ \gcd(f(p^e),q)> 1}} p^e, \quad\text{and}\quad B = n/A_1 A_2. \]
Then either $A_1>1$ or $A_2>1$. Moreover, every prime dividing $A_1$ exceeds $x^{1/U}$.

Suppose $A_2 > 1$. Since $A_2$ is squarefull, the number of $n\le x$ with $A_2 > x^{1/2}$ is $O(x^{3/4})$, which will be negligible for our purposes. So we assume that $A_2 \le x^{1/2}$. Given $A_2$, we  count the number of possibilities for the cofactor $A_1 B$. Note that $A_1 B\le x/A_2$ and that $A_1 B$ is free of prime factors $p \le x^{1/U}$ with $\gcd(f(p),q)>1$. So the sieve shows that the number of possibilities for $A_1 B$ is at most
\[ \frac{x}{A_2} \bigg(\prod_{\substack{p \le x^{1/U} \\ \gcd(f(p),q)>1}} \left(1-\frac{1}{p}\right)\bigg) \left(1 + O(\exp(-U/4))\right). \]
(We assume as usual that $q, x$ are large and $\frac{\log{q}}{\log{x}}, \delta(q)$ are small.) Since
\begin{equation*} \sum_{M \text{ squarefull}}\frac{1}{M} = \prod_{p} \left(1+\frac{1}{p^2}+\frac{1}{p^3}+\dots\right) 
= \frac{\zeta(2)\zeta(3)}{\zeta(6)} = 1.943\dots,
\end{equation*}
the count of $n$ with $A_2 > 1$ is bounded above by
\[ 0.945 x \bigg(\prod_{\substack{p \le x^{1/U} \\ \gcd(f(p),q)>1}} \left(1-\frac{1}{p}\right)\bigg) \left(1 + O(\exp(-U/4))\right). \]

Suppose now that $A_2=1$. Then $n=A_1 B$, where $A_1>1$ and every prime dividing $A_1$ exceeds $x^{1/U}$. Let $p$ be a prime dividing $A_1$, and write $A_1 =p S$. Then $n = pS B \le x$ where $S B \le x^{1-1/U}$. Given $S$ and $B$, the number of possible $p$ (and hence possible $n$) is, by Brun--Titchmarsh, at most
\[ \sum_{r \mid q} \sum_{\substack{p \le x/SB \\ F(p) \equiv 0 \pmod{r}}} 1 \ll_f \sum_{r \mid q} \frac{x}{r SB\log{(x/SBr)}} \ll \delta(q) U\frac{x}{\log{x}} \frac{1}{SB};  \]
here we have assumed that $q \le x^{1/2U}$, so that $x/SBr \ge (x/SB)/r \ge x^{1/2U}$ for every $r\mid q$. Summing on $S$ and $B$, the number of $n$ that arise is
\begin{multline*} \ll_f \delta(q) U\frac{x}{\log{x}} \bigg(\sum_{\substack{S \\ p \mid S \Rightarrow p \in (x^{1/U},x]}} \frac{1}{S}\bigg)\bigg( \sum_{\substack{B \le x\\ p \mid B,~ p \le x^{1/U}\Rightarrow \gcd(f(p),q)=1}}\frac{1}{B}\bigg) \\
\le \delta(q)  U\frac{x}{\log{x}} \bigg(\prod_{x^{1/U} < p \le x} \left(1-\frac{1}{p}\right)^{-2}\bigg) \bigg(\prod_{\substack{p \le x^{1/U} \\ \gcd(f(p),q)=1}}\left(1-\frac{1}{p}\right)^{-1}\bigg), \end{multline*}
which is 
\[ \ll \delta(q) U^3 \frac{x}{\log{x}} \prod_{p \le x^{1/U}} \left(1-\frac{1}{p}\right)^{-1} \prod_{\substack{p \le x^{1/U} \\ \gcd(f(p),q)>1}}\left(1-\frac{1}{p}\right) \ll \delta(q) U^2 x \prod_{\substack{p \le x^{1/U} \\ \gcd(f(p),q)>1}}\left(1-\frac{1}{p}\right). \]
But $\delta(q) = o(1)$, and so the final expression here is $o(x \prod_{p\le x^{1/U},~\gcd(f(p),q)>1}(1-1/p))$.

Collecting estimates shows that if  $U$ is fixed  sufficiently large, then eventually the number of $n\le x$ with $\gcd(f(n),q)=1$ exceeds
\[ \frac{1}{20}x \prod_{\substack{p\le x^{1/U}\\ \gcd(f(p),q)>1}}\left(1-\frac{1}{p}\right).\]
Bounding the product over $p \le x^{1/U}$ below by the product over $p\le x$ completes the proof.
 \end{proof}
 
Our second application of the sieve is an upper bound on the count of $n$ with few large prime factors. More precise results on this problem have been obtained by \cite{tenenbaum00}, but the comparatively simple Lemma \ref{lem:semismooth} below will suffice for our purposes.

Set $P_1^{+}(n) = P^{+}(n)$ and define, inductively, $P_{j+1}^{+}(n) =  P^{+}(n/P_1^{+}(n) \cdots P_j^{+}(n))$. Thus, $P_{j}^{+}(n)$ is the $j$th largest prime factor of $n$ (with multiple primes counted multiply), with $P_j^{+}(n)=1$ if $n$ has fewer than $j$ prime factors.
 
 \begin{lem}\label{lem:semismooth} Let $x\ge y \ge 10$. Let $J$ be an integer, $J\ge 2$. The number of $n\le x$ with $P_J^{+}(n)\le y$ is 
 \[ \ll_{J} x\frac{\log{y}}{\log{x}} (\log\log{x})^{J-1}. \]
 \end{lem}
 
\begin{proof}
Suppose that $P_J^{+}(n)\le y$ and write $n=AB$, where $A$ is the largest divisor of $n$ composed of primes not exceeding $y$. 
Then $\omega(B) \le \Omega(B) < J$. 

Clearly, $A \le x^{1/2}$ or $B\le x^{1/2}$. Suppose first that $A \le x^{1/2}$. Then $B \le x/A$ and $\omega(B)\le J-1$, so that by a classical theorem of Landau (see \cite[Theorem 437, p. \ 491]{HW08}), given $A$ there are $\ll \frac{x}{A \log{(x/A)}} (\log\log{(x/A)})^{J-2}\ll \frac{x}{A\log{x}}(\log\log{x})^{J-2}$ possible $B$. Summing $1/A$ on $A$ with $P^{+}(A) \le y$ introduces a factor $\prod_{p\le y}(1-1/p)^{-1} \ll \log{y}$, which yields for this case a slightly stronger upper bound than that claimed in the lemma.

Suppose now that $B \le x^{1/2}$. Since $A$ has no prime factors larger than $y$,
the sieve shows that given $B$, the number of possible $A\le x/B$ is $\ll \frac{x}{B} \prod_{y < p \le x^{1/2}}(1-1/p) \ll \frac{x}{B}\frac{\log{y}}{\log{x}}$. Since
\[ \sum_{\substack{B \le x \\ \omega(B) \le J-1}}\frac{1}{B} \le \sum_{j=0}^{J-1}\frac{1}{j!} \left(\sum_{p^e \le x}\frac{1}{p^e}\right)^j \ll_{J} (\log\log{x})^{J-1}, \]
the result follows.
\end{proof}
 
\subsection{Character sums of polynomials} We require estimates for (complete, multiplicative) character sums of polynomials modulo prime powers. For prime moduli, we use the following version of the Weil bound.
 
\begin{lem}\label{lem:weil} Let $\mathbb{F}_q$ be a finite field, and let $\chi_1,\dots,\chi_K$ be characters of $\mathbb{F}_q^{\times}$, extended to all of $\mathbb{F}_q$ by setting $\chi_k(0)=0$. Let $F_1(T),\dots,F_{K}(T)\in \mathbb{F}_q[T]$ be nonzero and pairwise relatively prime. Assume that for some $1\le k \le K$, the polynomial $F_k(T)$ is not an $\ord(\chi_k)$th power in $\mathbb{F}_q[T]$ or a constant multiple of such. Then 
$$ \left|\sum_{x \in \mathbb{F}_q} \chi_1(F_1(x)) \cdots \chi_{K}(F_{K}(x))  \right| \le (\sum_{k=1}^{K}d_k - 1)\sqrt{q},$$
where $d_k$ denotes the degree of the largest squarefree divisor of $F_k(T)$.
\end{lem}

Lemma \ref{lem:weil} is essentially Corollary 2.3 of \cite{wan97}. It is assumed in \cite{wan97} that all the $\chi_k$ are nontrivial, but this assumption is not used in the proof.

Estimating the sums to proper prime power moduli requires some stage setting. Let $p^m$ be an odd prime power, where $m\ge 2$. Let $g$ be a primitive root modulo $p^m$. Let $\chi$ be the Dirichlet character mod $p^m$ defined on integers $x$ coprime to $p$ by
\begin{equation}\label{eq:chidef} \chi(x) = \exp\left(2\pi \mathrm{i}\frac{ \mathrm{ind}_{g}(x)}{p^{m-1} (p-1)}\right), \end{equation}
where $g^{\mathrm{ind}_g(x)} \equiv x\pmod{p^m}$.

Let $F(T) \in \Z[T]$ be a nonconstant polynomial, and let $t$ be the largest nonnegative integer for which $p^t$ divides every coefficient of $F'(T)$.  Let $\tilde{F}(T)\in \mathbb{F}_p[T]$ denote the mod $p$ reduction of $p^{-t} F'(T)$. (Note that $\tilde{F}(T)$ is nonzero by the choice of $t$.) Let $\mathcal{A} \subset \mathbb{F}_p$ denote the set of roots of $\tilde{F}(T)$ in $\mathbb{F}_p$ that are not roots of the reduction of $F(T)$ mod $p$. For each $\alpha \in \mathcal{A}$, let $\nu_{\alpha}$ denote the multiplicity of $\alpha$ as a zero of $\tilde{F}(T)$, and let $M= \max_{\alpha \in \mathcal{A}} \nu_{\alpha}$. 

The following is an immediate consequence of Cochrane's Theorem 1.2 in \cite{cochrane02}; that very general result concerns mixed additive and multiplicative character sums, but see Theorem 2.1 of \cite{cochrane03} for the specialization  to  multiplicative character sums.

\begin{lem}\label{lem:cochrane} Under the above conditions, and the additional assumption that $m \ge t+2$,
we have
\[ \left|\sum_{x \bmod{p^m}} \chi(F(x))\right| \le (\sum_{\alpha \in \mathcal{A}} \nu_{\alpha}) p^{\frac{t}{M+1}} p^{m(1-\frac{1}{M+1})}.\]
\end{lem}

The proof of Theorem \ref{thm:joint} depends on the following consequence of Lemmas \ref{lem:weil} and \ref{lem:cochrane}, which seems of some independent interest.

\begin{prop}\label{prop:combinedcharsum} Let $F_1(T), \dots, F_K(T) \in \Z[T]$ be nonconstant and assume that the product $F_1(T) \cdots F_K(T)$ has no multiple roots. Let $p$ be an odd prime not dividing the leading coefficient of any of the $F_k(T)$ and not dividing the discriminant of $F_1(T) \cdots F_K(T)$. Let $m$ be a positive integer, and let $\chi_1,\dots,\chi_K$ be Dirichlet characters modulo $p^m$, at least one of which is primitive. Then \begin{equation}\label{eq:cancelation} \left|\sum_{x\bmod{p^m}} \chi_1(F_1(x)) \cdots \chi_K(F_K(x))\right| \le (D-1) p^{m(1-1/D)},\end{equation}
where $D = \sum_{k=1}^{K} \deg{F_k(T)}$.
\end{prop}

\begin{proof} Take first the case when $m=1$. When $D=1$, the left-hand side of \eqref{eq:cancelation} vanishes and \eqref{eq:cancelation} holds. When $D\ge 2$, we apply Lemma \ref{lem:weil} with $q=p$. The mod $p$ reductions of the $F_k(T)$ are nonzero (in fact, of the same degree as their counterparts in $\Z[T]$), and $F_1(T)\cdots F_K(T)$ is squarefree over $\mathbb{F}_p$, so that each $F_k(T)$ is squarefree and the $F_k(T)$ are pairwise relatively prime in $\mathbb{F}_p[T]$. Since some $\chi_k$ is primitive, it has order larger than $1$, and so $F_k(T)$ is not an $\mathrm{ord}(\chi_k)$th power in $\mathbb{F}_q[T]$ or a constant multiple of such. Lemma \ref{lem:weil} now yields \eqref{eq:cancelation}.

Henceforth, we suppose that $m\ge 2$. Let $g$ be a primitive root mod $p^m$, and let $\chi$ be the character mod $p^m$ defined in \eqref{eq:chidef}. We can write each $\chi_k$ in the form $\chi^{A_k}$, where $0 < A_k \le p^{m-1}(p-1)$. Then
\begin{equation}\label{eq:reexpression} \sum_{x\bmod{p^m}} \chi_1(F_1(x)) \cdots \chi_K(F_K(x)) = \sum_{x\bmod{p^m}}\chi(F(x)), \end{equation}
where 
\[ F(T) := F_1(T)^{A_1} \cdots F_K(T)^{A_K}. \]
Also,
\[ F'(T) = \left(\prod_{k=1}^{K} F_k(T)^{A_k-1}\right) G(T),\quad\text{where}\quad G(T):= \sum_{k=1}^{K} \Bigg(A_k F_k'(T) \prod_{\substack{1\le j \le K \\ j\ne k}}F_j(T)\Bigg).  \]
Let $t$ be the largest integer for which $p^t$ divides all the coefficients of $F'(T)$. Since none of the $F_k(T)$ are multiples of $p$, the power $p^t$ is also the largest power of $p$ dividing all the coefficients of $G(T)$ (by Gauss's content lemma). 

We claim that $t=0$. Choose, for each $k=1,\dots,K$, a root $\alpha_k$ of $F_k(T)$ from the algebraic closure $\overline{\mathbb{F}}_p$ of $\mathbb{F}_p$. Then in $\overline{\mathbb{F}}_p$, $$ G(\alpha_k)= (F_k'(\alpha_k) \prod_{\substack{1\le j \le K \\ j\ne k}}F_j(\alpha_k)) A_k,$$ and the factor in front of $A_k$ is nonzero. But if $t>0$, then $G(T)$ induces the zero function on $\overline{\mathbb{F}}_p$, forcing each $A_k$ to be a multiple of $p$. Then none of the $\chi_k$ are primitive characters mod $p^m$, contrary to hypothesis.

Now let $\mathcal{A}$, $\nu_\alpha$, and $M$ be defined as in the discussion preceding Lemma \ref{lem:cochrane}. Then each $\alpha \in \mathcal{A}$ is a root in $\mathbb{F}_p$ of the mod $p$ reduction of $G(T)$ of multiplicity $\nu_{\alpha}$. Moreover, $M \le \sum_{\alpha\in \mathcal{A}} \nu_{\alpha}\le \deg{G(T)} \le D-1$. The desired upper bound \eqref{eq:cancelation}  follows from \eqref{eq:reexpression} and Lemma \ref{lem:cochrane}.
\end{proof}

\section{Proof of Theorem \ref{thm:joint}} Throughout this proof, we  suppress the dependence of implied constants or implied lower/upper bounds on the constant $\epsilon>0$ as well as the family $f_1,\dots,f_K$. We let $F_1(T),\dots,F_K(T) \in \Z[T]$ be such that $f_k(p) = F_k(p)$ for all primes $p$.  We put \[ J:=(K+1)D \]  where, anticipating an application of Proposition \ref{prop:combinedcharsum},
\[ D := 1 + \sum_{k=1}^{K} \deg{F_k}(T). \]

It will be convenient to introduce the notation
\[ \sum\nolimits_{\ff}(x;q) := \sum_{\substack{n\le x \\ \gcd(f(n),q)=1}} 1. \]
Throughout this proof, when we say a term is \textsf{ignorable}, we mean that it is of smaller order than the right-hand side of \eqref{eq:joint}, that is, $o(\phi(q)^{-K} \sum_{\ff}(x;q))$.

By Lemma \ref{lem:coprimeasymp} (with $f=f_1\cdots f_K$) and the remark following it, we find that $$  \phi(q)^{-K} \sum\nolimits_{\ff}(x;q) \ge q^{-K} x(\log{x})^{o(1)} \ge x (\log{x})^{K\epsilon+o(1)}/\log{x} \ge x(\log{x})^{\epsilon+o(1)}/\log{x}.$$ So Lemma \ref{lem:semismooth} allows us to discard from the left-hand side of \eqref{eq:joint} those $n$ for which $P_J^{+}(n)\le L$, where
\[ L:= \exp((\log{x})^{\frac{1}{2}\epsilon}), \]
at the cost an ignorable error. Write each remaining $n$ in the form $n=m P_J \cdots P_1$, where each $P_j = P_j^{+}(n)$. We keep only those $n$ where $P^{+}(m) < P_J < \dots < P_1$. Any $n$ discarded at this step has a repeated prime factor exceeding $L$, and there are $O(x/L)$ of these, which is again ignorable. Note that for all of the remaining $n$, we have $f(n) = f(m) f(P_J) \cdots f(P_1)$, where each $P_j > L_m$ with
\[ L_m := \max\{P^{+}(m),L\}.\]

By the observations of the last paragraph, it suffices to prove that
\[ \sum\nolimits_{\ff}(x;q,\mathbf{a}) \sim \frac{1}{\phi(q)^K} \sum\nolimits_{\ff}(x;q), \]
where
\begin{align}\sum\nolimits_{\ff}(x;q,\mathbf{a}):&= \sum_{\substack{m\le x\\ \gcd(\prod_{k=1}^{K}f_k(m),q)=1}} \sum_{\substack{P_1, \dots, P_J\\ P_1\cdots P_J \le x/m \\ L_m < P_J < \dots < P_1 \\ (\forall k)~f_k(m)\prod_{j=1}^{J} f_k(P_j) \equiv a_k\pmod{q}}} 1 \notag\\
&=\sum_{\substack{m\le x\\ \gcd(\prod_{k=1}^{K}f_k(m),q)=1}}\frac{1}{J!}\sum_{\substack{P_1, \dots, P_J \text{ distinct}\\ P_1\cdots P_J \le x/m \\ \text{each $P_j > L_m$} \\ (\forall k)~f_k(m)\prod_{j=1}^{J} f_k(P_j) \equiv a_k\pmod{q}}} 1.\label{eq:sumfound}
\end{align}
We now remove the distinctness restriction in the final inner sum. Estimating crudely, this incurs  an error of size $O(x/mL)$ in the inner sum and an error of size $O(x\log{x}/L)$ in the double sum.  

For each $k=1,2,\dots,K$, let $u_k$ denote a value of $f_k(m)^{-1} a_k$ mod $q$ and define
\[ V_m := \{(v_1\bmod{q},\dots,v_J\bmod{q}): \gcd(v_1\cdots v_J,q)=1,~(\forall k)\, \prod_{j=1}^{J} F_k(v_j) \equiv u_k\pmod{q}\}. \]
Then writing $\mathbf{v} = (v_1\bmod{q},\dots, v_j\bmod{q})$,
\[ \sum_{\substack{P_1, \dots, P_J \\ P_1\cdots P_J \le x/m \\ \text{each $P_j > L_m$} \\ (\forall k)~f_k(m)\prod_{j=1}^{J} f_k(P_j) \equiv a_k\pmod{q}}} 1 = \sum_{\mathbf{v} \in V_m}  \sum_{\substack{P_1, \dots, P_J \\ P_1\cdots P_J \le x/m \\ \text{each $P_j > L_m$} \\ (\forall j)~P_j \equiv v_j \pmod{q} }} 1.  \]

For each $\mathbf{v} \in V_m$, we show how to remove the right-hand congruence  conditions on the $P_j$. First we handle $P_1$. Noting that $q\le (\log{x}) = (\log{L})^{2/\epsilon}$, the Siegel--Walfisz theorem (see, for example, \cite[Corollary 11.21]{MV07}) implies that for a certain positive constant $C=C_{\epsilon}$, 
\begin{align*} \sum_{\substack{P_1, \dots, P_J \\ P_1\cdots P_J \le x/m \\ \text{each $P_j > L_m$} \\ (\forall j)~P_j \equiv v_j \pmod{q} }} &1 = \sum_{\substack{P_2, \dots, P_J \\ P_2\cdots P_J \le x/m \\ \text{each $P_j > L_m$} \\ (\forall j\ge 2)~P_j \equiv v_j \pmod{q} }} \sum_{\substack{L_m < P_1 \le \frac{x}{mP_2\cdots P_J} \\ P_1 \equiv v_1\pmod{q}}} 1 \\
&\hskip-0.2in= \sum_{\substack{P_2, \dots, P_J \\ P_2\cdots P_J \le x/m \\ \text{each $P_j > L_m$} \\ (\forall j\ge 2)~P_j \equiv v_j \pmod{q} }} \left(\frac{1}{\phi(q)}\sum_{L_m < P_1 \le \frac{x}{mP_2\cdots P_J}} 1 + O\left(\frac{x}{m P_2\cdots P_J} \exp(-C\sqrt{\log{L}}) \right)\right) \\
&\hskip-0.2in= \frac{1}{\phi(q)}\sum_{\substack{P_1,P_2, \dots, P_J \\ P_1\cdots P_J \le x/m \\ \text{each $P_j > L_m$} \\ (\forall j\ge 2)~P_j \equiv v_j \pmod{q}}}1 + O\left(\frac{x}{m} \exp\left(-\frac{1}{2}C\sqrt{\log{L}}\right)\right).
\end{align*}
In the same way, the congruence conditions on $P_2,\dots,P_J$ can be removed successively to yield
\[  \sum_{\substack{P_1, \dots, P_J \\ P_1\cdots P_J \le x/m \\ \text{each $P_j > L_m$} \\ (\forall j)~P_j \equiv v_j \pmod{q} }} 1 = \frac{1}{\phi(q)^J}\sum_{\substack{P_1,P_2, \dots, P_J \\ P_1\cdots P_J \le x/m \\ \text{each $P_j > L_m$}}} 1 + O\left(\frac{x}{m} \exp\left(-\frac{1}{2}C\sqrt{\log{L}}\right)\right).\] The main term on the right-hand side is independent of $\mathbf{v}$. Keeping in mind that $\#V_m \leq q^J \le (\log{x})^J$ for all $m$, we deduce from \eqref{eq:sumfound} that
\begin{equation}\label{eq:sumf} \sum\nolimits_{\ff}(x;q,\mathbf{a}) =  \sum_{\substack{m\le x\\ \gcd(\prod_{k=1}^{K}f_k(m),q)=1}}\frac{\#V_m}{\phi(q)^J}  \cdot \frac{1}{J!}\sum_{\substack{P_1, \dots, P_J \\ P_1\cdots P_J \le x/m \\ \text{each $P_j > L_m$}}}1  + O\left(x \exp\left(-\frac{1}{4}C\sqrt{\log{L}}\right)\right).\end{equation}
To handle the main term, notice that
\[ \sum_{\substack{P_1, \dots, P_J \\ P_1\cdots P_J \le x/m \\ \text{each $P_j > L_m$} \\ \text{some $\gcd(f(P_j),q)>1$}}} 1 \le J \sum_{p \mid q} \sum_{\substack{P_1, \dots, P_J \\ P_1\cdots P_J \le x/m \\ \text{each $P_j > L_m$} \\ p \mid f(P_1)}} 1.   \]
The condition that $p\mid f(P_1)$ puts $P_1$ in a certain (possibly empty) set of $O(1)$ residue classes mod $p$. Removing these congruence condition by the Siegel--Walfisz theorem (exactly as above) we find that (with $C$ as above)
\[ \sum_{\substack{P_1, \dots, P_J \\ P_1\cdots P_J \le x/m \\ \text{each $P_j > L_m$} \\ p \mid f(P_1)}} 1 \ll \frac{1}{p}\sum_{\substack{P_1, \dots, P_J \\ P_1\cdots P_J \le x/m \\ \text{each $P_j > L_m$}}} 1 + \frac{x}{m} \exp\left(-\frac{1}{2}C\sqrt{\log{L}}\right) \]
and so
\[ \sum_{\substack{P_1, \dots, P_J \\ P_1\cdots P_J \le x/m \\ \text{each $P_j > L_m$} \\ \text{some $\gcd(f(P_j),q)>1$}}} 1 \ll \delta(q)\sum_{\substack{P_1, \dots, P_J \\ P_1\cdots P_J \le x/m \\ \text{each $P_j > L_m$}}} 1 + \frac{x}{m} \exp\left(-\frac{1}{4}C\sqrt{\log{L}}\right).  \]
Since $\delta(q)=o(1)$, 
\begin{align*}
\sum_{\substack{P_1, \dots, P_J \\ P_1\cdots P_J \le x/m \\ \text{each $P_j > L_m$}}}1
 &= (1+O(\delta(q))) \sum_{\substack{P_1, \dots, P_J \\ P_1\cdots P_J \le x/m \\ \text{each $P_j > L_m,~ \gcd(f(P_j),q)=1$}}}1 + O\left(\frac{x}{m} \exp\left(-\frac{1}{4}C\sqrt{\log{L}}\right)\right) \\
 &=  (1+O(\delta(q))) J! \sum_{\substack{P_J < \dots < P_1 \\ P_1\cdots P_J \le x/m \\ \text{each $P_j > L_m,~\gcd(f(P_j),q)=1$}}}1 + O\left(\frac{x}{m} \exp\left(-\frac{1}{4}C\sqrt{\log{L}}\right)\right).
\end{align*}
The following claim will be established at the end of this section as an application of Proposition  \ref{prop:combinedcharsum}.

\textbf{Claim.} $\#V_m \sim q^{J}/\phi(q)^K$, uniformly in $m$.

We insert the estimate of the Claim, together with the last display, into \eqref{eq:sumf}.  Since $\delta(q)=o(1)$, we have $\frac{q^J}{\phi(q)^J} (1+O(\delta(q))) = 1+o(1)$. We find that up to an ignorable error, $\sum\nolimits_{\ff}(x;q,\mathbf{a})$ is equal to 
\[ (1+o(1)) \frac{1}{\phi(q)^K} \sum_{\substack{m\le x \\ \gcd(\prod_{k=1}^{K}f_k(m),q)=1}} 
\sum_{\substack{L_m < P_J < \dots < P_1 \\ P_1\cdots P_J \le x/m \\ \text{each $\gcd(f(P_j),~q)=1$}}}1. \]
We can view the double sum as counting  those numbers $n\le x$ with $\gcd(f(n),q)=1$ and certain extra constraints: Namely, the $J$th largest prime factor of $n$ exceeds $L$ and none of the largest $J$ prime factors are repeated. But (by reasoning seen at the start of this proof) dropping the extra constraints incurs an ignorable error. So up to an ignorable error, $\sum_{\ff}(x;q,\mathbf{a})$ is equal to $\frac{(1+o(1))}{\phi(q)^K} \sum_{\ff}(x;q)$. By  definition of \textsf{ignorable},  $\sum_{\ff}(x;q,\mathbf{a}) \sim \frac{1}{\phi(q)^K} \sum_{\ff}(x;q)$, and we have seen already that this suffices to complete the proof of Theorem \ref{thm:joint}.
\begin{proof}[Proof of the Claim.]
Using $\chi_0$ for the trivial character mod $q$, orthogonality yields
\begin{align}
\phi(&q)^K\#V_m \notag \\&=  \sum_{\chi_1,\dots,\chi_K\bmod{q}} \bigg(\prod_{k=1}^{K} \bar{\chi}_k(u_k)\bigg)\bigg(\sum_{x_1,\dots,x_J \bmod{q}} \chi_0(\prod_{j=1}^{J} x_j) \chi_1(\prod_{j=1}^{J}F_1(x_j))\cdots \chi_K(\prod_{j=1}^{J}F_K(x_j))\bigg) \notag\\ &= \sum_{\chi_1,\dots,\chi_K\bmod{q}} \bigg(\prod_{k=1}^{K} \bar{\chi}_k(u_k)\bigg) S_{\chi_1,\dots,\chi_K}^{J}, \label{eq:ssum}
\end{align}
where
\[ S_{\chi_1,\dots,\chi_K}:= \sum_{x\bmod{q}} \chi_0(x) \chi_1(F_1(x)) \cdots \chi_K(F_K(x)). \]
The number of $x$ mod $q$ where one of $x, F_1(x),\dots,F_K(x)$ has a common factor with $q$ is $\ll q \delta(q) = o(q)$, and so the tuple  $\chi_1,\dots\chi_K$ of trivial characters makes a contribution $\sim q^{J}$ to \eqref{eq:ssum}. So to complete the proof, it suffices to show that
\begin{equation}\label{eq:notalltrivial} \sum_{\substack{\chi_1,\dots,\chi_K \bmod{q}\\\text{not all trivial}}} |S_{\chi_1,\dots,\chi_K}|^{J} \end{equation} 
has size $o(q^J)$.

Assume that $\chi_1,\dots,\chi_K$ are Dirichlet characters mod $q$, not all of which are trivial.
Factor $q= \prod_{p\mid q}p^{e_p}$. Each character $\chi_k$, for $k=0,1,\dots,K$, admits a unique decomposition of the form $\chi_k = \prod_{p\mid q} \chi_{k,p}$, where $\chi_{k,p}$ is a Dirichlet character modulo $p^{e_p}$. By the \textsf{type} of the tuple $\chi_1,\dots,\chi_K$, we mean the $\omega(q)$-element sequence of positive integers  $\{\mathfrak{f}_p\}_{p\mid q}$, where each 
\[ \mathfrak{f}_p = \mathrm{lcm}[\mathfrak{f}(\chi_{1,p}),\dots,\mathfrak{f}(\chi_{K,p})]. \]

Write $q=q_0 q_1$, where $q_1$ is the unitary divisor of $q$ supported on the primes $p\mid q$ for which $\mathfrak{f}_p > 1$.  Note that $q_1>1$, since not all of $\chi_1,\dots,\chi_K$ are trivial. By the Chinese remainder theorem,
\[ S_{\chi_1,\dots,\chi_K} =  \prod_{p\mid q} \bigg(\sum_{x\bmod{p^{e_p}}}\chi_{0,p}(x) \chi_{1,p}(F_1(x)) \cdots \chi_{K,p}(F_K(x))  \bigg),\]
from which we see that 
\begin{align*} |S_{\chi_1,\dots,\chi_K}| &\le q_0 \prod_{p\mid q_1} \bigg|\sum_{x\bmod{p^{e_p}}}\chi_{0,p}(x) \chi_{1,p}(F_1(x)) \cdots \chi_{K,p}(F_K(x))  \bigg| \\
&= q_0 \prod_{p\mid q_1} \frac{p^{e_p}}{\mathfrak{f}_p}  \bigg|\sum_{x\bmod{\mathfrak{f}_p}}\chi_{0,p}(x) \chi_{1,p}(F_1(x)) \cdots \chi_{K,p}(F_K(x))  \bigg|.  \end{align*} At least one of $\chi_{1,p},\dots,\chi_{K,p}$ has conductor $\mathfrak{f}_p$, and so the remaining sum on $x$ may be estimated by Proposition \ref{prop:combinedcharsum}, yielding
$|S_{\chi_1,\dots,\chi_K}|\le q (D-1)^{\omega(q_1)} \prod_{p\mid q_1} \mathfrak{f}_p^{-1/D}$. (If none of the $F_k(T)$ are multiples of $T$, we apply Proposition \ref{prop:combinedcharsum} with the polynomials $T, F_1(T), \dots, F_k(T)$; otherwise, the sum on $x$ is unchanged if we remove the term $\chi_{0,p}(x)$ and we apply the proposition with $F_1(T), \dots, F_k(T)$. Keep in mind that since $\delta(q)=o(1)$, all the prime factors of $q$ are large, so the nondivisibility conditions on $p$ in Proposition \ref{prop:combinedcharsum} are certainly satisfied.) Hence (since $J=(K+1)D$)
$|S_{\chi_1,\dots,\chi_K}|^{J} \le q^{J} (D-1)^{\omega(q_1) J} \prod_{p\mid q_1} \mathfrak{f}_p^{-(K+1)}$. 
There are no more than $(\prod_{p\mid q_1} \mathfrak{f}_p)^K$ tuples $\chi_1,\dots,\chi_K$ sharing this type, so that the contribution from all such tuples to  \eqref{eq:notalltrivial} is 
at most $q^{J} (D-1)^{\omega(q_1)J} \prod_{p\mid q_1} \mathfrak{f}_p^{-1}$. 
Summing  $\mathfrak{f}_p$ over all powers of $p$, for $p \mid q_1$, reveals that the contribution from all types corresponding to a given $q_1$ is at most
\[  q^{J} (D-1)^{\omega(q_1)J} \frac{q_1}{\phi(q_1)} \prod_{p\mid q_1} p^{-1} \le q^{J} (D-1)^{\omega(q_1)J} 2^{\omega(q_1)} \prod_{p\mid q_1} p^{-1}. \]
Finally, summing over all unitary divisors $q_1$ of $q$ with $q_1>1$ bounds \eqref{eq:notalltrivial} by
\[ q^J\left(\prod_{p \mid q}\left(1+\frac{2(D-1)^J}{p}\right)-1\right) \le q^J(\exp(2(D-1)^J\delta(q))-1) = o(q^J). \]
Collecting estimates completes the proof of the Claim.\end{proof}

\providecommand{\bysame}{\leavevmode\hbox to3em{\hrulefill}\thinspace}
\providecommand{\MR}{\relax\ifhmode\unskip\space\fi MR }
\providecommand{\MRhref}[2]{%
  \href{http://www.ams.org/mathscinet-getitem?mr=#1}{#2}
}
\providecommand{\href}[2]{#2}

\end{document}